\documentclass{article}

\usepackage{amsmath, amsthm, amssymb}
\usepackage{color,xcolor}
\newtheorem{lemma}{Lemma}

\newtheorem{proposition}[lemma]{Proposition}
\newtheorem{remark}[lemma]{Remark}

\newtheorem{theorem}[lemma]{Theorem}

\newcommand{\e}{\varepsilon} 

\title{Non-local singular perturbations \\ of non-convex functionals -- recent results}
\author{Andrea Braides\footnote{email: {\tt braides@mat.uniroma2.it}}\\ 
{\small Department of Mathematics, University of Rome Tor Vergata,}\\ {\small via della ricerca scientifica 1, Rome (Italy)}}
\date{}
\begin{document} 

\maketitle
\begin{abstract} 
Singular perturbations have been used to select solutions of (non-convex) variational problems with a multiplicity of minimizers. The prototype of such an approach is the gradient theory of phase transitions by L.~Modica, who specialized some earlier Gamma-convergence results by himself and S.~Mortola contained in a seminal paper, validating the so-called minimal-interface criterion. I will give an overview of some recent results on perturbations with fractional and higher-order seminorms both in the framework of phase transitions and of free-discontinuity problems, relating these results with the Bourgain-Brezis-Mironescu and Maz'ya-Shaposhnikova limit analysis for fractional Sobolev seminorms, and with the theory of Gamma-expansions.
\smallskip

{\bf MSC codes:} 49J45, 35B25, 82B26, 35R11.

{\bf Keywords:} $\Gamma$-convergence, non-convex energies, higher-order fractional Sobolev spaces, phase-transition problems, free-discontinuity problems.

\end{abstract}

\section{Introduction}
I begin with some personal recollections about the early times of $\Gamma$-convergence, for which singular perturbation problems were a prototypical issue. When I started my degree thesis work under the guidance of Ennio De Giorgi in 1981, the theoretical foundations of $\Gamma$-convergence had been laid and it was already a stable theory. I got my personal idea of what had been the developments of this theory in the preceding years from those who directly worked with De Giorgi at the first developments, from Tullio Franzoni, to Sergio Spagnolo, to Luciano Modica, who was the first to give a course on $\Gamma$-convergence at Pisa University in the academic year 1982-1983. 

An example that De Giorgi liked to quote was (one-dimensional) elliptic homogenization, with an oscillating coefficient taking two values (he called it the `Greek function' as he would freely associate its graph to an hellenic ornamental motive). Indeed, $\Gamma$-convergence had directly stemmed from the interest in homogenization problems for elliptic operators, that De Giorgi and Spagnolo interpreted as a convergence of the energies. Even though $\Gamma$-convergence was then born as a variational counterpart of $G$-convergence, the convergence of Green's operators for elliptic problems, it was clear to De Giorgi that the flexibility of a purely topological convergence made the new notion not bound to a specific class of problems or structure of the functionals (as had been Mosco convergence for convex functionals). Beside the theoretical developments leading to the paper in collaboration with Franzoni and successive efforts to explore all possible variants of the convergence, De Giorgi was keen on obtaining results showing how $\Gamma$-convergence was not restricted to the usual applications to homogenization, since it needed no regularity, convexity or even coerciveness assumptions. I recall the paper by Buttazzo and Dal Maso \cite{MR527419}, where an example is given of convergence of non-convex and non-equi-Lipschitz integral functionals with a loss of smoothness, and also my thesis work was about the development of a homogenization theory for non-coercive energies \cite{MR766686}. In this context the work of Modica and Mortola \cite{MM} was aimed at showing how the type of functional obtained as a $\Gamma$-limit may be completely different from the converging ones. The 
functionals they examine have the form
\begin{equation}
F_\e(u)=\frac1\e\int_\Omega W\Big(\frac u\e\Big)dx+ \e \int_\Omega|\nabla u|^2dx,
\end{equation}
with the explicit choice $W(z)=\sin^2(\pi z)$, but their result is valid for any $1$-periodic non-negative continuous function $W$ vanishing only on $\mathbb Z$. The corresponding $\Gamma$-limit in the $L^1$-topology is then simply
\begin{equation}\label{2}
F_0(u)=m_W\|Du\|(\Omega),\qquad m_W=2\int_0^1\sqrt{W(s)}ds
\end{equation}
with domain $BV(\Omega)$ and with $\|Du\|(\Omega)$ denoting the total variation of $Du$ on $\Omega$. A second result in the paper concerns the functionals 
\begin{equation}\label{3}
F_\e(u)=\frac1\e\int_\Omega W(u)dx+ \e \int_\Omega|\nabla u|^2dx,
\end{equation}
which differ from the first ones in the first term, now forcing in the limit $u$ to take integer values.
The $\Gamma$-limit has the same form of the first one, but with domain $BV(\Omega;\mathbb Z)$ 
the set of $BV$ functional taking values in $\mathbb Z$. 
Even though Modica and Mortola justify their analysis with possible applications to minimal surfaces, the application of their (second) result to phase-transition problems had to wait for almost a decade.

\bigskip
Singular perturbations are ofter used to obtain approximate solutions to problems which possess no solution or infinitely many ones, or to give a selection criterion among different solutions. One can think of the theory of viscosity solutions of non-linear PDE as an example. In a variational setting, a simple way to obtain solutions is by adding terms with (higher-order) derivatives, whose form is often dictated by modeling considerations. The prototypical example is
$$
\int_\Omega W(u)dx+\e^2\int_\Omega|\nabla u|^2dx,
$$
with $W$ a non-convex potential. The limit of solutions $u_\e$ to specific problems obtained using this energies are particular solutions to the corresponding problem with $\e=0$ characterized by additional properties, the type of properties depending on the non-convexity of $W$, which can be of {\em double-well} type, or with a {\em non-convexity at infinity}. 

Potentials of double-well type $W$ are such that their convex envelope differs from $W$ in bounded intervals, in which case, up to adding a linear term, we can consider $W$ with two minimizers.  This is the case of the Cahn--Hilliard theory of phase separations \cite{CH}, for which a {\em minimal-interface criterion} was conjectured (e.g.~by Gurtin  \cite{gurtin}) stating that limit of minima satisfying (suitable) volume constraints take value only in the set of the two minimizers of $W$ and the corresponding two sets are characterized as having a minimal interface of separation. In order to properly state such a convergence result one has to use the theory of {\em sets of finite perimeter} and of functions of bounded variation. The proof of the minimal-interface criterion was then achieved by Modica \cite{Modica} using the $\Gamma$-convergence arguments by Modica and Mortola applied to energies \eqref{3}, reinterpreting their limit as an interfacial energy, which can be written as
$$
m_W {\rm Per}(E,\Omega);
$$
that is, the {\em perimeter} of the set $E$ in $\Omega$, weighted by the {\em surface tension} $m_W$. The reason why the set $E$ appears is that sequences $u_\e$ with equibounded $F_\e(u_\e)$ are precompact in measure, so that by the boundedness of the first term in \eqref{3} their limit $u$ satisfies $u(x)\in\{-1,1\}$ and $u=-1+2\chi_E$, where $E=\{x: u(x)=1\}$.

\smallskip
Potentials with a non-convexity at infinity have been treated in a different functional setting, considering perturbations depending on $\nabla u$ instead of perturbations depending on $u$. The simplest case is when $W$ is the truncated quadratic potential $W(z)=\min\{|z|^2,1\}$. After some scaling arguments, the perturbed functionals  have the form
\begin{equation}\label{4}
F_\e(u)=\frac1\e\int_\Omega W(\sqrt\e |\nabla u|)dx+ \e^3 \int_\Omega|\nabla^2 u|^2dx,
\end{equation}
defined on $H^2(\Omega)$. The $\Gamma$-limit of these functionals have been studied by Alicandro et al \cite{ABG} and Bouchitt\'e at al \cite{bouds} in the framework of the approximation of {\em free-discontinuity problems}. The $\Gamma$-limit is defined in the space of {\em special functions of bounded variation}, where it
takes the form
\begin{equation}\label{4}
F(u)=\int_\Omega |\nabla u|^2 dx+ m^\infty_2 \int_{S(u)\cap\Omega}\sqrt{|u^+(x)-u^-(x)|}d\mathcal H^{d-1},
\end{equation}
where $S(u)$ is the set of {\em discontinuity points} of $u$ and $u^\pm$ are the traces of $u$ on both sides of $S(u)$. Note that if $d=1$ then we can take as a variable $v=u'$ and reread $F$ as a functional defined on spaces of measures, with domain measures $\mu$ which admit a decomposition as a sum of a part absolutely continuous with respect to the Lebesgue measure and a series of Dirac deltas. In order to avoid the description in spaces of measures we use the variable $u$ through its gradient. 

\smallskip
Recent developments have made it possible to extend the theory  of singular perturbations to the case with fractional derivatives of any order. In the case of double-well potentials, $k$ integer and $s\in (0,1)$, we study energies
\begin{equation}\label{m+k-0}
\int_\Omega W(u)dx+ \e^{2(k+s)} \int_{\Omega\times \Omega}\frac{|\nabla^{k}u(x)-\nabla^{k}u(y)|^2}{|x-y|^{d+2s}}dx\,dy.
\end{equation}
In order to obtain a (local) perimeter functional, the correct scaling is obtained by dividing by $\e$, and considering the functionals 
\begin{equation}\label{m+k}
F^{k+s}_\e(u)=\frac1\e\int_\Omega W(u)dx+ \e^{2(k+s)-1} \int_{\Omega\times \Omega}\frac{|\nabla^{k}u(x)-\nabla^{k}u(y)|^2}{|x-y|^{d+2s}}dx\,dy.
\end{equation}
with $k+s>\frac12$, so that we still have a singular perturbation after scaling.
The limit is still a perimeter functional, which is in general of the form
\begin{equation}\label{m+k-limit}
F^{k+s}(E)=\int_{\partial E}\varphi_{k+s}(\nu)d\mathcal H^{d-1},
\end{equation}
where $\nu$ is the normal to $E$ at the points in $\partial E$. In the one-dimensional case $d=1$ the characterization of the perimeter reduces to that of the surface tension $m_{k+s}$. This paper examines the recent and less recent contributions by Alberti et al.~\cite{ABS}, Fonseca and Mantegazza \cite{FM}, Savin and Valdinoci \cite{SV}, Palatucci and Vincini \cite{PaVi}, Brusca et al. \cite{BDS}, Solci \cite{Solci}, Picerni \cite{Picerni}, Braides et al.~\cite{BOP} in a unified setting, showing in particular how a common line of proof can be followed, that is also compatible with the extension to the integers case, formally $s=0$ and $k\ge 1$, and also to the limit degenerate case $s=\frac12$. 

The case of potentials with a non-convexity at infinity is much more complex since it requires the description of an interfacial energy $\varphi$ also depending on the discontinuity size $|u^+(x)-u^-(x)|$. In a very complex work, using delicate localized interpolation inequalities Solci \cite{MR4937467} has dealt with the higher-dimensional integer case  
\begin{equation}\label{m+k}
F^{k}_\e(u)=\frac1\e\int_\Omega W(u)dx+ \e^{2k-1} \int_{\Omega} \|\nabla^{k}u\|^2dx,
\end{equation}
where $k\ge 2$ is an integer and $\|\cdot\|$ is the operator norm, showing that the limit is 
\begin{equation}\label{4}
F(u)=\int_\Omega |\nabla u|^2 dx+ m^\infty_k \int_{S(u)}\root k \of {|u^+(x)-u^-(x)|}d\mathcal H^{d-1},
\end{equation}
with a proper characterization of $m^\infty_k$. The case of fractional perturbations or norms different from the operator norm is almost completely open.
The content of this paper is a (very condensed) part of a course held in SISSA, Trieste in the 2024.

\section{Singular perturbations of double-well energies}
We focus on the one-dimensional case $d=1$.
The double-well potential $W\colon\mathbb R\to [0,+\infty)$ is a continuous function satisfying

\smallskip i) $W(z)=0$ if and only if $z\in\{-1,1\}$;

\smallskip ii) there exist the second derivatives $W''(-1)$ and $W''(1)$ and are strictly positive;

\smallskip iii) $\liminf\limits_{|z|\to+\infty} W(z)>0$.

\smallskip The scaled perturbed functionals take the form 
 \begin{eqnarray}\label{k+s}\nonumber
F^{k+s}_\e(u)&=&F^{k+s}_\e(u, I)\\
&=&\frac1\e\int_I W(u)dx+ \e^{2(k+s)-1} \int_{I\times I}\frac{|u^{(k)}(x)-u^{(k)}(y)|^2}{|x-y|^{1+2s}}dx\,dy,
\end{eqnarray}
with domain the fractional Sobolev space $H^{k+s}(I)$ of the functions $u\in H^k(I)$ such that the double integral in \eqref{k+s} (the square of the corresponding Gagliardo seminorm of $u^{(k)}$) is finite; $I$ is a bounded open interval of $\mathbb R$. We refer to \cite{leofrac} for an introduction to these fractional Sobolev spaces.

A key feature of energies \eqref{k+s} is their behaviour under scaling. More precisely, if we set $v(t)=u(\e t)$ then we have
 \begin{equation}\label{k+s-risc}
F^{k+s}_\e(u)=F^{k+s}_1(v, \tfrac1\e I)=\int_{\frac1\e I} W(v)dx+ \int_{\frac1\e I\times \frac1\e I}\frac{|v^{(k)}(x)-v^{(k)}(y)|^2}{|x-y|^{1+2s}}dx\,dy.
\end{equation}
After this change of variables we can work on functionals depending on $\e$ only through their domains. For convenience of notation, we write $F^{k+s}$ in the place of $F^{k+s}_1$.

We will show that $F^{k+s}_\e$ are equi-coercive with respect to the convergence in measure and their limit is a one-dimensional sharp-interface functional $F^{k+s}$, with domain piecewise-constant functions taking only the values $-1$ and $1$, which we can identify with $BV(I;\{-1,1\})$. On such functions, this functional is simply
 \begin{equation}\label{k+s-lim}
F^{k+s}(u)= m_{k+s} \#(S(u)),
\end{equation}
determined by the non-zero {\em surface tension} $m_{k+s}$; $S(u)$ denotes the set of discontinuity points of $u$.

\subsection{Coerciveness in measure}
For the classical Modica--Mortola result, coerciveness can be obtained by estimating the number of {\em transition intervals}; that is, with fixed  $\eta\in(0,\tfrac12)$, intervals $(x_1,x_2)$ such that $x_1=-1+\eta$ and $x_2=1-\eta$.
This is obtained by noting that the term $\int_{x_1}^{x_2}|u'|^2dx$ is at least $\tfrac1{|x_2-x_1|}$, so that the energy carried by such an interval is at least $ C_\eta\tfrac{|x_2-x_1|}\e+\tfrac\e{|x_2-x_1|}\ge 2\sqrt{C_\eta}>0$. This implies that $|x_2-x_1|\sim\e$ and that the number of such transitions is equi-bounded. Applying the Bolzano--Weierstrass Theorem we can suppose that actually the transition intervals are independent of $\e$, and the compactness follows by noting that if $u_\e>1-\eta$ (or $u_\e<1-\eta$) on a fixed interval then $u_\e$ tends to $1$ (or to $-1$, respectively). Note that the optimization of this argument actually proves the lower bound with $c_W$.  

In the general fractional case, if $k\ge 1$ this argument localized on transition intervals is not possible, simply considering functions $u_\e$ affine on a transition interval for which the Gagliardo seminorm is $0$. The compactness argument must then use some non-local properties, which,  loosely speaking, provide some extra boundary conditions to transition intervals, and is separate from the computation of the lower bound. It relies on the following local interpolation result by Solci \cite[Lemma 6]{Solci}.

\begin{lemma}[local interpolation lemma]\label{lemmac}
If $F^{k+s}_\e(u_\e)\le S<+\infty$ for all $\e>0$, then there exists $\overline \eta>0$ such that for all $\eta\in (0,\overline\eta)$ there exists a constant $L$ such that for all intervals $J_\e$ such that $|J_\e|\ge \e L$ and $||u_\e|-1|\le \eta$ there exists at least a point $t\in J_\e$ such that $|u_\e^{(\ell)}(t)|\le \e^{-\ell} \eta$ for all $\ell\in\{1,\ldots, k-1\}$.
\end{lemma}

\begin{proof} The proof follows by interpolation inequalities for fractional Sobolev spaces applied to either the function $u_\e-1$ or $u_\e+1$, after noting that by assumption (ii) on $W$ there exist positive constants $\alpha, \beta,\overline\eta$ such that $\alpha(|z|-1)^2\le W(z)\le \beta (|z|-1)^2$ if $||z|-1|<\overline\eta$. 
\end{proof}

We have the following compactness theorem  \cite[Theorem 2]{Solci}.

\begin{theorem}[Compactness]\label{compth} 
Let $\{u_\e\}$ be a family in $H^{k+s}(I)$ such that $F_\e(u_\e)\le S<+\infty$ for all $\e$. Then, there exist $u\in BV(I;\{-1,1\})$ such that, up to subsequences, $u_{\e}\to u$ in measure. 
 Furthermore, if  $W$ satisfies $W(z)\ge c_1|z|^p-c_2$ for some $c_1,c_2$ and $p\ge 1$, then $u_{\e}\to u$ in $L^p(I)$.
\end{theorem} 
 
 \begin{proof} As already observed, it suffices to show that, for fixed $\eta$, $u_\e$ has a finite number of transitions between $1-\eta$ and $-1+\eta$. 
Suppose otherwise, then for all $\e$ small enough there exists an interval $J_\e=(a_\e, b_\e)$ containing a transition and such that $F_\e(u_\e, J_\e)=o(1)$ as $\e\to 0$, and $b_\e-a_\e> k L\e$. 
If $J_\e$ contains up to $k$ intervals such that $||u_\e|-1|<\eta$ then we can apply Lemma \ref{lemmac}.
We deduce that the functions $v_\e(t)= u_\e(a_\e +t(b_\e-a_\e))$ tend to a constant in $H^{s+k}(0,1)$, since they have vanishing $H^{k+s}$ seminorm and $v^{(\ell)}_\e$ is equibounded at a point. This gives a contradiction with the boundary conditions transition.

If $J_\e$ contains more such intervals then there exists an interval of length of order $\e$ with a large number of transitions. Applying Rolle's theorem iteratively, we have the existence of points where $u^{(\ell)}(t)=0$ for all $\ell\in\{1,\ldots,k-1\}$, and we can repeat the argument by contradiction.

Finally, if $W(z)\ge c_1|z|^p-c_2$ then  $\int_{\{|u_\e|>2\}}|u_\e|^p dx\to 0$ and $(u_\e\vee{-2})\wedge 2$ converges to $u$ in $L^p(I)$, which proves that $u_\e\to u$  in $L^p(I)$.
\end{proof}

\begin{remark}\label{limitat}\rm
We note that in the proof of the compactness theorem instead of $-1+\eta$ and $1-\eta$ we can take any two numbers $\lambda_1,\lambda_2$ such that $[\lambda_1,\lambda_2]\cap\{-1,1\}=\emptyset$. In particular, for every $r>0$ we have that the number of transitions between $1+r$ and $1+2r$ is equibounded, and analogously for transitions between $-1-2r$ and $-1-r$.
\end{remark}

\subsection{The surface tension}
The explicit characterization of the surface tension $m_1=c_W$  (see \eqref{2}) in the Modica--Mortola result cannot be exported for other perturbations. It is convenient to characterize it in terms of an optimal-profile problem as in \cite{FT}, where it is shown that
\begin{eqnarray}\label{emme1} \nonumber
m_1&=&\inf\Big\{\int_{\mathbb R} (W(v)+|v'|^2)dt: v(\pm\infty)=\pm1\Big\}
\\
&=&\inf_{T>0} \Big\{\int_{-T}^T (W(v)+|v'|^2)dt: v(\pm T)=\pm1\Big\}.
\end{eqnarray}
Note that the second formula also gives an upper bound, suggesting the construction of recovery sequences of the form $u_\e(x)=v(\tfrac{x-\overline x}\e)$ or $u_\e(x)=-v(\tfrac{x-\overline x}\e)$ on intervals $(\overline x-\e T,\overline x+\e T)$ with $\overline x\in S(u)$ for $u\in BV(I;\{-1,1\})$.

 This characterization of the surface tension can be exported to the fractional case as in the following proposition.

\begin{proposition}[$k+s$-surface tension]\label{sut}
Let $k\in \mathbb N$ and $s\in (0,1)$ with $k+s>\frac12$. Then we have the equality
\begin{eqnarray}\label{emme}\nonumber
&&m_{k+s} =\inf\Big\{\int_{\mathbb R} W(v)dx+\int_{\mathbb R^2}\frac{|v^{(k)}(x)-v^{(k)}(y)|^2}{|x-y|^{1+2s}} dx\,dy :\\
&& \hskip4.5cm v\in H^{k+s}_{\rm loc}(\mathbb R), \lim_{x\to\pm\infty}v(x)=\pm1\Big\}\\
\label{emmetilde}\nonumber
&&\hskip1cm =\inf_{T>0} \inf\bigg\{\int_{-T}^{+T} W(v)dt +\int_{\mathbb R^2}\frac{|v^{(k)}(x)-v^{(k)}(y)|^2}{|x-y|^{1+2s}} dx\,dy :\\
&& \hskip4cm v\in H^{k+s}_{\rm loc}(\mathbb R), v(x)={\rm sign}\, x\hbox{ \rm if }|x| \ge T\Big\}.
\end{eqnarray}
\end{proposition}

In order to prove this result and the lower bound for the $\Gamma$-limit, we preliminarily state a technical cut-off  lemma, taken from the proof of \cite[Proposition 12]{Solci}. In the proof of the following lemma the cut-off functions $\varphi$ are constructed in such a way that the resulting $v$, defined on intervals $(-T,T)$, are constant in a `safe zone' of length $L$ at the boundary of the interval. This is necessary in order to ensure that the additional contribution obtained by extending $v$ to the whole $\mathbb R$ is small for $L$ large. 

\begin{lemma}[A cut-off lemma]\label{culemma}
Let $\eta,\sigma,T,L>0$, with $\eta<1$, $u\in H^{k+s}(-T,T)$ such that $|u-1|\le \eta$ on $[\sigma,T]$ and $|u+1|\le \eta$ on $[-T,-\sigma T]$. Then there exists $v\in H^{k+s}_{\rm loc}(\mathbb R)$ such that $v(x)={\rm sign}(x)$ if $|x|\ge T-L$ and for all $\delta>0$ there exists $C_\delta$ such that
\begin{eqnarray}\label{lemma-co}\nonumber
F^{k+s}(v,\mathbb R))&\le& (1+\delta) F^{k+s}(u,(-T,T))\\ \nonumber
&&+ \frac{C_\delta}{L^{2s}}\Big(F^{k+s}(u,(T-2L,T))
+ F^{k+s}(v,(-T,-T+2L))\\
&&+\frac1{\sigma T}+(1+\sigma^{2s} T^{2s})F^{k+s}(u,(-T,T))\Big) +o(1)
\end{eqnarray}
as $L\to+\infty$.
\end{lemma}

\begin{proof}
The proof relies on a cut-off argument. The function $v$ is defined by setting
$$
v(x)= u(x)\,\varphi\big(\tfrac{|x|-T+2L}L\big)+\big({\rm sign}(x)-\varphi\big(\tfrac{|x|-T+2L}L\big)\big),
$$
where $\varphi\in C^\infty(\mathbb R)$ with $0\le \varphi\le 1$ is such that $\varphi(t)=1$ if $t\le 0$, $\varphi(t)=0$ if $t>1$. 
The lemma follows after some lengthy computations, for which we refer to \cite[Proposition 12]{Solci}.  The computations are particularly delicate since the estimate of $F^{k+s}(v,(-T,T))$ must be subdivided in many double integrals, for whose estimates different arguments must be used. An important point is to note that it is convenient to use that $|\varphi^{(\ell)}(x)-\varphi^{(\ell)}(y)|\le \frac{C}{L^\ell}|x-y|$ for $|x-y|\le 2L$ for $\ell\in\{0,\ldots, k-1\}$, while $|\varphi^{(\ell)}(x)-\varphi^{(\ell)}(y)|=0$ if $|x-y|\ge 2L$ for $\ell\in\{1,\ldots, k-1\}$.
\end{proof}

\begin{proof}[Proof of Proposition \rm\ref{sut}]  Since the expression in \eqref{emme} gives a lower value than that in 
\eqref{emmetilde} as all test functions in \eqref{emmetilde} are also test functions for  \eqref{emme}, it is sufficient to prove the converse inequality. To that end, $\eta>0$ with fixed we consider a minimizer $u$ for \eqref{emme} up to $\eta$. We note that by Remark \ref{limitat} applied to $u_\e(t)=u(\frac{t}\e)$, the number of the transitions of $u$ between $-1+\eta$ and $1-\eta$ is bounded, as well as those between $1+\eta$ and $1+2\eta$ and $-1-2\eta$ and $-1-\eta$. Moreover, the total measure of the set where $||u|-1|>\eta$ is finite. Hence, up to a translation, there exists $x_0$ such that $|u(x)-{\rm sign}x|<\eta$ for $|x|>x_0$ and $F^{k+s}(x_0,+\infty)+F^{k+s}(-\infty,-x_0)\le \eta$.
We fix $\sigma>0$ and define $T=x_0/\sigma$. We can apply Lemma \ref{culemma} with $L=T/3$, obtaining 
\begin{eqnarray*}\label{lemma-co}\nonumber
F^{k+s}(v,\mathbb R))&\le& (1+\delta) F^{k+s}(u,(-T,T))\\ \nonumber
&&+ \sigma^{2s}\frac{C_\delta}{x_0^{2s}}\Big(1+\frac1{x_0}+(1+x_0^{2s})F^{k+s}(u,(-T,T))\Big) +o(1),
\end{eqnarray*}
as $\sigma\to 0$, and the claim follows by the arbitrariness of $\sigma$ and $\delta$.
\end{proof}

\subsection{The $\Gamma$-limit}
By the compactness theorem above, the functionals $F^{k+s}_\e$ are equi-coercive with respect to the convergence in measure. Hence, the corresponding  $\Gamma$-limit result can be stated only on $BV(I;\{-1,1\})$.

\begin{theorem}[sharp-interface limit] Let $k\in \mathbb N$ and $s\in (0,1)$ with $k+s>\frac12$; then
we have 
 \begin{equation}\label{k+s-lim-1}
F^{k+s}(u)= m_{k+s} \#(S(u)).
\end{equation}
for all $u\in BV(I;\{-1,1\})$.
\end{theorem}

\begin{proof} Let $u_\e\to u$ in measure with $F^{k+s}_\e(u_\e)\le S<+\infty$. Let $t_0\in S(u)$, and define 
$\widetilde u_\e(t) = u_\e(t_0+\e t)$. We have $F^{k+s}(\widetilde u_\e, (-T,T))\le S$ for all $T>0$, and argue similarly to the proof of  Proposition \rm\ref{sut}. Given $\eta>0$, we fix $x_0$ such that, up to a change of sign, $|\widetilde u_\e(x)-{\rm sign}(x)|\le \eta$ for $|x|\in[x_0, T]$. We can construct $v_\e\in H^{k+s}_{\rm loc}(\mathbb R)$ such that $v_\e(x)={\rm sign}(x)$ if $|x|\ge \frac23 T$ and for all $\delta>0$ there exists $C_\delta$ such that 
\begin{eqnarray*}\nonumber
F^{k+s}(v_\e,\mathbb R))\le (1+\delta) F^{k+s}_\e(u_\e,(t_0-\e T,t_0+\e T)) +\frac{C_\delta}{T^{2s}}.
\end{eqnarray*}
Since $v_\e$ is a test function for $m_{k+s}$, summing up for $t_0\in S(u)$ and taking into account that all $(t_0-\e T,t_0+\e T)$ are disjoint for $\e$ small enough, we obtain that
$$
m_{k+s}\,\#(S(u))\le  (1+\delta) \liminf_{\e\to 0} F^{k+s}_\e(u_\e) + \frac{C_\delta}{T^{2s}}.
$$
By the arbitrariness of $T$ and $\delta$ we obtain the liminf inequality. 

\smallskip
The construction of an approximate recovery sequence for a function $u\in BV(I;\{-1,1\})$ follows a usual pattern:
given $\eta>0$ we find $T>0$ and $v\in H^{k+s}_{\rm loc}(\mathbb R)$ such that $v(x)={\rm sign}\, x$ if $|x| \ge T$ and 
$$
F^{k+s}(v,\mathbb R)\le m_{k+s}+\eta.
$$
Let $S(u)=\{t_1,\ldots, t_N\}$ and define 
$$
u_\e(t)=\begin{cases}
v\Big(\frac{u(t_j^+)-u(t_j^-)}2\cdot\frac{t-t_j}\e\Big) & \hbox{ if } |t-t_j|\le \e\\
u(t) & \hbox{ otherwise.}
\end{cases}
$$
We set $x_0=\inf I$, $x_j=(t_{j+1}+t_j)/2$ for $j\in\{1,\ldots, N-1\}$ and $x_N=\sup I$.
In order to check that $u_\e$ is a recovery sequence, we need to show that the amount of energy due to
the interactions between $x$ and $y$ belonging to intervals $(x_{j_1-1}, x_{j_1})$ and $(x_{j_2-1}, x_{j_2})$ with $j_1\neq j_2$. If $k\ge 1$ we note that actually the interaction is only between the intervals $(t_{j_1}-\e T,t_{j_1}+\e T)$ and $(t_{j_2}-\e T,t_{j_2}+\e T)$, so that $|x-y| \sim |t_{j_2}-t_{j_1}|$ independently of $\e$, which gives that these contributions are asymptotically negligible. In the case $k=1$ one additionally has to note that for ${j_2}={j_1}+1$ there is no interaction between pair of points in $(t_{j_1}+\e T, t_{j_2}-\e T)$. This implies that
$$
F_\e^{k+s}(u_\e,I)\le (m_{k+s}+\eta)\#(S(u)) +o(1)
$$
as $\e\to 0$, which proves the claim by the arbitrariness of $\eta$.
\end{proof}

\subsection{Analysis at critical exponents}
The results in the previous section show that the sharp-interface limit is described by the surface tension, which then defines a function $t\mapsto m_t:=m_{k_t+s_t}$ for $t\in (\frac12,+\infty)\setminus\mathbb N$, where $k_t= \lfloor t\rfloor$
and $s_t= t-\lfloor t\rfloor$ are the integer part and the fractional part of $t$, respectively. The behaviour of $m_t$ at $\frac12$ and at $\mathbb N$ is singular, in the sense that the limit of  $m_t$ is $+\infty$ at these points. In this section we briefly recall how this singular behaviour can be `corrected', reporting results by Picerni \cite{Picerni} and Solci \cite{}.

\subsubsection{Analysis close the exponent $1/2$}
If $k=0$ and $s=1/2$, we have $2(k+s)-1=0$, so that the functional in \eqref{m+k} 
$$
\frac1\e\int_I W(u)dx+ \int_{I\times I}\frac{|u(x)-u(y)|^2}{|x-y|^{2}}dx\,dy
$$
presents a second term which is non-singular. Note that the second term is finite on $H^{1/2}(I)$, while the first term forces $u$ to take the values $-1$ and $1$, which is an incompatible requirement.

The second term in the functional above is scale invariant. 
As is often the case for scale-invariant energies, the correct scaling turns out to be logarithmic,
and we define
\begin{equation}\label{m+k1/2}
F^{1/2}_\e(u)=\frac1{\e|\log\e|}\int_I W(u)dx+ \frac1{|\log\e|}\int_{I\times I}\frac{|u(x)-u(y)|^2}{|x-y|^{2}}dx\,dy.
\end{equation}
A classical result describing the asymptotic behaviour of $F^{1/2}_\e$ is the following (see \cite{ABS,SV}, to which we refer for a proof; see also \cite{PaVi}).

\begin{theorem}
The functionals $F^{1/2}_\e$ are equi-coercive  with respect to the convergence in measure, and their
$\Gamma$-limit of $F^{1/2}_\e$ as $\e\to 0$ is the sharp-interface functional with domain $BV(I;\{-1,1\})$ 
$$
F^{1/2}(u)= m_{1/2} \#(S(u))
$$
with $m_{1/2}=8$.
\end{theorem}

Despite having a different scaling, functionals $F^{1/2}_\e$ can be related to $F^s_\e$, using a more complex scaling factor than $\frac1\e$. This factor can be computed by examining the minimum problems for $F^s_\e$ at varying $\e$ and $s>1/2$. The following result is proved in \cite{Picerni}.

\begin{theorem}[correction at $s=\frac12^+$]
The scaled functionals 
\begin{eqnarray}\nonumber&&
\hskip-1.1cm\widetilde F^s_\e(u):={\frac{2s-1}{1-\e^{2s-1}}} F^s_\e(u)\\
&&={\frac{2s-1}{\e(1-\e^{2s-1})}}\Big(\int_I W(u(x))dx +\e^{2s}\int_I\int_I \frac{|u(x)-u(y)|^2}{|x-y|^{1+2s}} dx\,dy\Big)
\end{eqnarray}
are equicoercive  with respect to the convergence in measure, and they
$\Gamma$-converge to $F^{1/2}$ for all $s=s_\e\to\frac12^+$. In particular, we have that
$$
\lim_{s\to1/2^+} (2s-1)m_s=m_{1/2}.
$$
\end{theorem}

\begin{remark}[interpolation between logarithmic and linear regimes]\rm The factor $\frac{2s-1}{1-\e^{2s-1}}$ can be seen as an interpolation between the constant
$2s-1$ (with $\e\to 0$ at $s$ fixed) and the logarithmic scaling, since for fixed $\e>0$ we have
$$
\frac{2s-1}{1-\e^{2s-1}}= \frac{2s-1}{1-e^{-(2s-1)|\log\e|}}\sim  \frac{1}{|\log\e|}
$$
as $s\to\frac12^+$. The asymptotic result can be interpreted as a $\Gamma$-expansion \cite{BT},
which shows that
$$\hskip-.2cm
\int_I W(u(x))dx +\e^{2s}\int_I\int_I \frac{|u(x)-u(y)|^2}{|x-y|^{1+2s}}dx\,dy\sim
\frac{8\e(1-\e^{2s-1})}{2s-1}\#(S(u))
$$
on $BV(I;\{-1,1\})$ as $s\to\frac12^+$ and $\e\to 0$.
\end{remark}

\subsubsection{Analysis at positive integers from the left}
A result by Bourgain, Brezis and Mironescu \cite{BBM} shows that the functionals 
$$
(1-s) \int_{I\times I}\frac{|v(x)-v(y)|^2}{|x-y|^{1+2s}}dx\,dy
$$
$\Gamma$-converge to $\int_I |v'|^2dx$ as $s\to 1^-$. In particular if we do not multiply by the vanishing term $1-s$ the $\Gamma$-limit results finite only on constant functions, from which we obtain that $m_s$ tends to $+\infty$ as $s\to 1^-$, showing a discontinuity with the Modica result. This discontinuity can be eliminated once we use the scaling by $1-s$, for which it can be shown (see \cite{Solci}) that the limit of the {\em scaled surface tensions} 
\begin{eqnarray}\label{emmetil1}\nonumber
&&\widetilde m_{s} =\inf\Big\{\int_{\mathbb R} W(v)dx+(1-s)\int_{\mathbb R^2}\frac{|v(x)-v(y)|^2}{|x-y|^{1+2s}} dx\,dy :\\
&& \hskip4.5cm v\in H^{s}_{\rm loc}(\mathbb R), \lim_{x\to\pm\infty}v(x)=\pm1\Big\}
\end{eqnarray}
converge to $m_1$ as in \eqref{emme1} as $s\to 1^-$. In order to extend such a correction to all positive integers, we have to first make sure that a phase-transition limit exists also for (local) integer derivatives of order $k>1$. This has been proved by Fonseca and Mantegazza \cite{FM} for second derivatives and only recently by Brusca, Donati and Solci \cite{BDS}.

\begin{theorem}[phase transitions with integer-order perturbations] Let $W$ satisfy the same hypotheses as in the results already stated, and let $k\ge 1$ be a positive integer. Then the functionals
 \begin{equation}\label{k}
F^{k}_\e(u)=\frac1\e\int_I W(u)dx+ \e^{2k-1} \int_{I}{|u^{(k)}(x)|^2}dx,
\end{equation}
with domain $H^k(I)$ are equicorcive in measure, and their $\Gamma$-limit with respect both to the convergence on measure and to the $L^2$-convergence is the sharp-interface functional $m_k\#(S(u))$ with domain $BV(I;\{-1,1\})$, where
\begin{eqnarray}\label{emmek}\nonumber
&&m_{k} =\inf\Big\{\int_{\mathbb R} W(v)dx+\int_{\mathbb R}|v^{(k)}|^2dx:v\in H^{k}_{\rm loc}(\mathbb R), \lim_{x\to\pm\infty}v(x)=\pm1\Big\}\\
\label{emmetildek}\nonumber
&&\hskip1cm =\inf_{T>0} \inf\bigg\{\int_{-T}^{+T} W(v)dt +\int_{-T}^T{|v^{(k)}(x)|^2}dx\ :
v\in H^{k}(-T,T), \\
&&\hskip1.2cm v(\pm T)=\pm1 \hbox{ \rm and } v^{(\ell)}(x)=0 \hbox{ \rm if }|x|= T, \ell\in\{1,\ldots,k-1\}\Big\}.
\end{eqnarray}
\end{theorem}

\begin{proof} It suffices to note that the line of proof for the energies $F^{k+s}_\e$ works almost unchanged, with actually some simplifications in the computations when using a cut-off argument.
\end{proof}

\begin{remark}[enhanced boundary conditions]\rm
We note that, contrary to \eqref{emmetilde}, the second formula in \eqref{emmetildek} uses minimum problems on the bounded intervals $(-T,T)$ with the homogeneous boundary conditions on the derivatives. Contrary to the case $k=1$, this has to be achieved using local interpolation inequalities  as in Lemma \ref{lemmac} \cite{Solci,BDS}.
\end{remark}

In order to correct the behaviour of surface tensions at integer points `from the left' we define the Bourgain-Brezis-Mironescu scaled surface tensions
\begin{eqnarray}\label{emmeks}\nonumber
&&m^{BBM}_{k+s} =\inf\Big\{\int_{\mathbb R} W(v)dx+(1-s)\int_{\mathbb R^2}\frac{|v^{(k)}(x)-v^{(k)}(y)|^2}{|x-y|^{1+2s}} dx\,dy :\\
&& \hskip4.5cm v\in H^{k+s}_{\rm loc}(\mathbb R), \lim_{x\to\pm\infty}v(x)=\pm1\Big\}.
\end{eqnarray}
Note that for $s$ fixed the surface tension describes the $\Gamma$-limit for the corresponding scales functionals.
To check this, it suffices to note that $(1-s)^{-1}W$ is a potential satisfying the same assumptions as $W$. We then have the following result.

\begin{proposition}[continuity from the left] For all $k\ge 0$ non-negative integers we have $\lim\limits_{s\to 1^-}m^{BBM}_{k-1+s}=m_k$.
\end{proposition}

\begin{proof} The proof relies on the Bourgain, Brezis and Mironescu result applied with $u^{(k-1)}$ in the place of $v$, and on the use of Lemma \ref{culemma}, which guarantees that, up to a uniformly small error, minimum problems for $m^{BBM}_{k-1+s}$ can be taken with a fixed large $T$ independent of $s$.
\end{proof}

\subsubsection{Analysis at positive integers from the right}
The result by Bourgain, Brezis and Mironescu has a counterpart as $s\to 0^+$ proved by Maz'ya and Shaposhnikova \cite{Ms}. In this case the functionals 
\begin{equation}\label{MS-s}
s\int_{I\times I}\frac{|v(x)-v(y)|^2}{|x-y|^{1+2s}}dx\,dy
\end{equation}
$\Gamma$-converge to $ 2\int_I |v|^2dx$ as $s\to 0^+$ on functions $u$ with compact support in $I$. 
This restriction on the support is not relevant for our analysis since, by the restriction $k+s>\frac12$, the result is applied to functions $v=u^{(k)}$ with $k\ge 1$, with the condition that $u$ is constant (either $-1$ or $1$) outside $I$, and hence $v$ is can be assumed with compact support in $I$. 
Again, if we do not multiply by the vanishing term $s$, the $\Gamma$-limit results finite only on constant functions, from which we obtain that $m_{k+s}$ tends to $+\infty$ as $s\to 0^+$, showing a discontinuity with the result by Brusca, Donati, and Solci. This discontinuity can be eliminated once we use the scaling by $s$ in the following way.

For $k\ge 1$ and $s\in (0,1)$, we define the Maz'ya and Shaposhnikova scaled surface tensions
\begin{eqnarray}\label{emmeks-MS}\nonumber
&&m^{MS}_{k+s} =\inf\Big\{\int_{\mathbb R} W(v)dx+\frac{s}2\int_{\mathbb R^2}\frac{|v^{(k)}(x)-v^{(k)}(y)|^2}{|x-y|^{1+2s}} dx\,dy :\\
&& \hskip4.5cm v\in H^{k+s}_{\rm loc}(\mathbb R), \lim_{x\to\pm\infty}v(x)=\pm1\Big\}.
\end{eqnarray}
The normalization factor $2$ is due to the factor appearing in the $\Gamma$-limit of \eqref{MS-s} as $s\to 0$.
Again note that for $s$ fixed the surface tension describes the $\Gamma$-limit for the corresponding scaled functionals, after remarking that $2s^{-1}W$ is a potential satisfying the same assumptions as $W$. 
We then have the following result.

\begin{proposition}[continuity from the right] For all $k\ge 1$ non-negative integers we have $\lim\limits_{s\to 0^+}m^{MS}_{k+s}=m_k$.
\end{proposition}

\begin{proof} The proof is the same as for $s\to 1^-$ but using the Maz'ya and Shaposhnikova result applied with $u^{(k)}$ in the place of $v$.
\end{proof}

\begin{remark}[continuous interpolations for $k+s\ge \frac12$]\rm
Taking into account the `corrections' for $k\ge 1$ and $s\to 0^+$ and $s\to 1^-$ at the same time, one can construct functionals
$$
\int_I W(u)dx+ \e^{2(k+s)}\alpha_s\int_{\mathbb R^2}\frac{|v^{(k)}(x)-v^{(k)}(y)|^2}{|x-y|^{1+2s}} dx\,dy 
$$
that interpolate uniformly between functionals \eqref{k} in the sense of expansion by $\Gamma$-convergence \cite{BT}.  In \cite{Solci} the choice is $\alpha_s=s(1-s)2^{s-1}$. 
\end{remark}

\subsection{Some notes on the higher-dimensional case}
The higher-dimensional case, when we replace the interval $I$ with an open set $\Omega$ in $\mathbb R^d$ and the functionals have the form 
 \begin{equation}\label{k+s-hd}
F^{k+s}_\e(u)=\frac1\e\int_\Omega W(u)dx+ \e^{2(k+s)-1} \int_{\Omega\times \Omega}\frac{\|\nabla^ku(x)-\nabla^ku(y)\|^2}{|x-y|^{d+2s}}dx\,dy
\end{equation}
is more technical. We note however that the compactness result follows from the one-dimensional result by Solci.

\begin{proposition}[equi-coerciveness]
Let $k\ge 0$, $s\in (0,1)$ with $k+s>\frac12$, and let $u_\e$ be such that $F^{k+s}_\e(u_\e)\le S<+\infty$.
We assume that the potential $W$ also satisfies a $2$-polynomial growth condition.
Then, up to subsequences, $u_\e$ converges in $L^2(\Omega)$ to some $u\in BV(\Omega;\{-1,1\})$.
\end{proposition}

\begin{proof}
The proof follows from the one-dimensional result by a characterization by slicing of higher-order fractional Sobolev spaces \cite{leofrac}. We refer to \cite{BOP} for details. 
\end{proof}

Identifying a function $u\in BV(\Omega;\{-1,1\})$ with the set $A=\{u=1\}$ such that $u=-1+2\chi_A$, the $\Gamma$-limit of $F^{k+s}_\e$ can be written as a perimeter functional 
$$
F_{k+s}(A)=\int_{\partial^*A}\varphi_{k+s}(\nu_A)d{\mathcal H}^{d-1},
$$
with $\partial^*A$ the reduced boundary of $A$ and $\nu_A$ the internal normal to $A$.
The integrand $\varphi_{k+s}$ is a function such that its positively homogeneous extension of degree $1$ is convex (see \cite{AmBra}), and can be characterized through some asymptotic formulas \cite{BOP,MR3748585}.
The form of $\varphi_{k+s}$ depends on the interplay between the norm $\|\cdot\|$ of a $k$-th order tensor and the non-locality of the Gagliardo seminorm, except for the case $k=0$, treated in \cite{SV}, where we have the norm of a $0$-th order tensor; that is, of a scalar quantity. 

With respect to the one-dimensional case, the comparison with the local perturbation by the $k$-th derivative  seems more complex. That local case is part of the work of Brusca, Donati and Solci \cite{BDS}, when $\|\cdot\|$ is the operator norm and an argument by sections can be applied, and has been generalized to arbitrary norms by Brusca, Donati and Trifone \cite{BDT}. In particular, the `correction' of the non-local energies at integer points seems  to require new arguments, at least when the limit energy density $\varphi_{k}$ derived in the local case is not a constant.

\section{Singular perturbations of energies with non-convexity at infinity}
The second type of non-convex potentials $W$ are those with a non-convexity at infinity; that is, such that the  convex envelope of $W$ does not coincide with $W$ on one or two half lines. Typical potentials with a well at infinity are the Lennard-Jones potential and the Perona-Malik potential $W(z)=\log(1+z^2)$.  For such potentials the results are still partial and sparse. In this section we analyze the few known results trying to draw some parallel with the previous sections. Moreover, we will mainly consider a one-dimensional environment.

\subsection{Perturbations of local type}
Up to a translation, it is not restrictive to suppose that $W$ has a minimum point in $0$ and that $W(0)=0$.  
A prototypical $W$ is the {\em truncated quadratic potential} $W(z)=\min\{z^2,1\}$. Since most of the issues are open and independent of the shape of the potential, we will restrict to using this particular potential.

As remarked in the Introduction, it is convenient to use $u'$ as a variable in the place of $u$, and suitably scale the energies, in order to avoid the necessity of defining the $\Gamma$-limit on spaces of measures.
We begin with the perturbation of minimal (integer) order; that is, with a second derivative. After scaling, the energies we consider have the form 
\begin{equation}\label{sc-W}
\int_I \frac1\e W(\sqrt\e u')dx+\e^3 \int_I|u''|^2dx.
\end{equation}
The power $3$ is in accord with the scaling arguments already used for higher-order derivatives, as $3=2k-1$, with  $k=2$.
The reason why we consider the scaling in \eqref{sc-W} for the first integral is that, assuming that $W''(0)>0$, we have 
$$
\frac1\e W(\sqrt\e u')=\frac1\e\Big(\frac12 W''(0)\e |u'|^2 + o(\e|u'|^2)\Big)=\frac12 W''(0) |u'|^2 + o(|u'|^2),
$$
which suggests in a way that $u'\in L^2(I)$. Note that the scaling is not the unique choice (see e.g.~\cite{GP-2}), but in a sense is the most complete and justified as an expansion by $\Gamma$-convergence \cite{BT}.
For the prototypical quadratic energies 
\begin{equation}
F^2_\e(u)=\int_I \min\Big\{|u'|^2,\frac1\e\Big\}dx+\e^3 \int_I|u''|^2dx 
\end{equation}
indeed, this turns out to be true for the approximate differential of $u$. Moreover, 
sequences with equi-bounded energy are precompact with respect to the $L^1$-convergence (up to addition of constants), and the $\Gamma$-limit of $F^2_\e$ in this topology is 
$$
F_2(u)=\int_\Omega|\nabla u|^2dx+ m_2\sum_{S(u)} \sqrt{|u^+(x)-u^-(x)|}
$$
defined in the space $SBV(I)$ (see \cite{bln}). This has been proved by  Alicandro, Braides, and Gelli \cite{ABG} and by Bouchitt\'e, Dubs, and Seppecher \cite{BDS}.

The constant $m_2$ represents the energy necessary to produce a jump discontinuity of size $1$. We reduce to this case by a $1/2$-homogeneity argument. It is important to note that $m_2$ can be interpreted as given by an optimal-profile problem
$$
m_2=\inf_T\min\bigg\{ T+\int_0^T|v''|^2dt: v(0)=0, v(T)=1, v'(0)=v'(T)=0\bigg\},
$$
where the transition between $0$ and $1$ is optimized on an interval of length $T$, which is also a variable of the problem. 
This minimum problem is obtained after scaling by $1/\e$ an interval where the modulus of the derivative of a function $u_\e$ is {\em above the threshold} $1/\sqrt\e$. After defining $v(t)=u(\e t)$, we obtain the homogeneous boundary conditions on $v'$ at the endpoints as $\e\to 0$. This consideration draws a parallel with the phase-transition energies,  in which we have observed that boundary conditions in \eqref{emme1} are automatically determined in the Modica--Mortola case, while this is not the case for higher-order perturbations. 

\smallskip
The general result for perturbations with local higher-order energies is the following theorem by Solci \cite{MR4937467}.

\begin{theorem}[integer higher-order perturbations]
Let $k\ge 2$, $k\in\mathbb N$, and let 
\begin{equation}\label{Fkfd}
F^k_\e(u)=\int_I \min\Big\{|u'|^2,\frac1\e\Big\}dx+\e^{2k-1} \int_I|u^{(k)}|^2dx 
\end{equation}
be defined on $H^k(I)$. Then the functionals are equi-coercive, up to addition of constants, with respect to the convergence in $L^1(I)$, the domain of the limit is a subset of $SBV(I)$, where it takes the form 
$$
F_k(u)=\int_\Omega|\nabla u|^2dx+ m_k\sum_{S(u)} \root k \of{|u^+(x)-u^-(x)|},
$$
and the constant $m_k$ is defined by 
\begin{eqnarray}\label{cpj}\nonumber
m_k&=&\inf_T\min\bigg\{ T+\int_0^T|v^{(k)}|^2dt: v(0)=0, v(T)=1,\\
&& \hskip 1cm v^{(\ell)}(0)=v^{(\ell)}(T)=0 \hbox{ for } \ell\in\{1,\ldots, k-1\}\Big\}.
\end{eqnarray}
\end{theorem}

The proof of this result is extremely delicate, and we refer to \cite{Solci} for the details of the proof. Again, one issue is, as for phase-transition energies, to enhance the boundary conditions, showing that one can add also conditions on the derivatives of order higher than the first. This can be done using interpolation inequalities, possibly enlarging the intervals where functions have to be considered `above the threshold'. Unfortunately, the length of these intervals itself concurs in the determination of $m_k$, and the corresponding minimum problems can be extremely sensitive to the change of length if the jump size is small, when the $1/k$-homogeneity argument hinted at above is applied. This second issue can be solved by fine constructions when we have small jumps, and by the use of relaxation results. 

\begin{remark}[the higher-dimensional case]\rm
In \cite{MR4937467} the higher-dimensional case has been treated for a perturbation using the operator norm of the tensor of the $k$-th order derivatives, obtaining  $F_k$ with the corresponding form as in  \eqref{Fkfd}. The case of norms different from the operator one is open. Note in particular that it is not clear how to generalize the formula for $m_k$ to obtain a function $\varphi_k$, which now may depend not only on the normal to $S(u)$ but also to the jump size. We note that the result with the operator norm is sufficient to obtain equi-coerciveness in the general case by comparison.
\end{remark}

\subsection{A result for fractional perturbations}
While the case of general higher-order fractional perturbations seems technically complex, we can treat the case  when  $k=1$ and $s\in(0,1)$, which may be regarded as a variation of the result in \cite{ABG}. The functionals 
$$
F^{1+s}_\e(u)=\int_I \min\Big\{|u'|^2,\frac1\e\Big\}dx+\e^{2s +1}\int_I\int_I \frac{|u'(x)-u'(y)|^2}{|x-y|^{1+2s}}dx\,dy,
$$
defined on $H^{1+s}(I)$, have been studied by Braides and Vitali \cite{BV}. Note that the exponent $2s+1=2(k+s)-1$ with $k=1$ is given by the usual scaling.

With \eqref{cpj} in mind, the corresponding constant $m_{1+s}$ can be written in the form of \eqref{emmetilde}
with the non-local part integrated on the whole line, as
\begin{eqnarray*}
&&m_{1+s}=\inf_T\min\bigg\{ T+\int_{\mathbb R}\int_{\mathbb R} \frac{|v'(x)-v'(y)|^2}{|x-y|^{1+2s}} dx\,dy:  v\in H^{1+s}_{\rm loc}(\mathbb R)\ \qquad\ \qquad\ \\ 
&&\hskip4cm v(t)=0 \hbox{ for }t\le 0, v(t)=1 \hbox{ for }t\ge T\Big\}.
\end{eqnarray*}
With this definition, we can describe the asymptotics of $F^{1+s}_\e$ as $\e\to 0$ as follows.

\begin{theorem} The functionals $F^{1+s}_\e$ are equi-coercive, up to addition of constants, with respect to the $L^1(I)$-convergence. The corresponding $\Gamma$-limit is 
$$
F_{1+s}(u)= 
\int_I|u'|^2dx+ m_{1+s}\int_{S(u)} |u^+(x)-u^-(x)|^{\frac1{1+s}}\, d{\mathcal H}^{d-1},
$$
defined on $SBV(I)$, where 
\end{theorem}

For a proof we refer to \cite{BV}. We only note that if $s>\frac12$ the boundary conditions $v'(0)=v'(T)=0$ are (meaningful and) enough to obtain compactness following the argument of \cite{ABG}, but limiting the double integral to $(0,T)\times (0,T)$ in the definition of $m_{1+s}$ gives a sub-optimal lower bound.

\begin{remark}[some open problems]\rm
Beside the extension to higher-order fractional derivatives and to higher dimension, we note that it could be interesting to analyze the behaviour at integer points, in particular starting from the one-dimensional case with $k=1$ and $s$ tending to $0$ or tending to $1$. Note that in the first case we expect a problem of linear growth defined on $BV(I)$.
\end{remark}

\noindent{\bf Acknowledgements.}
The author gratefully acknowledges the active participation of the students of the course at SISSA. The author is a member of GNAMPA of INdAM.

\bibliographystyle{abbrv}

\bibliography{references}

\begin{thebibliography}{10}

\bibitem{ABS}
G.~Alberti, G.~Bouchitt\'e, and P.~Seppecher.
\newblock Un r\'esultat de perturbati\-ons singuli\`eres avec la norme
  {$H^{1/2}$}.
\newblock {\em C. R. Acad. Sci. Paris S\'er. I Math.}, 319(4):333--338, 1994.

\bibitem{ABG}
R.~Alicandro, A.~Braides, and M.~S. Gelli.
\newblock Free-discontinuity problems generated by singular perturbation.
\newblock {\em Proceedings of the Royal Society of Edinburgh Section A:
  Mathematics}, 128(6):1115--1129, 1998.

\bibitem{AmBra}
L.~Ambrosio and A.~Braides.
\newblock Functionals defined on partitions in sets of finite perimeter. {II}.\
  {S}emicontinuity, relaxation and homogenization.
\newblock {\em J. Math. Pures Appl. (9)}, 69(3):307--333, 1990.

\bibitem{bouds}
G.~Bouchitt{\'e}, C.~Dubs, and P.~Seppecher.
\newblock Regular approximation of free-discontinuity problems.
\newblock {\em Mathematical Models and Methods in Applied Sciences},
  10(07):1073--1097, 2000.

\bibitem{BBM}
J.~Bourgain, H.~Brezis, and P.~Mironescu.
\newblock Another look at {S}obolev spaces.
\newblock In {\em Optimal {C}ontrol and {P}artial {D}ifferential {E}quations},
  pages 439--455. IOS, Amsterdam, 2001.

\bibitem{MR766686}
A.~Braides.
\newblock Homogenization of noncoercive integrals.
\newblock {\em Ricerche Mat.}, 32(2):347--368, 1983.

\bibitem{bln}
A.~Braides.
\newblock {\em Approximation of Free-discontinuity Problems}.
\newblock Springer-Verlag, Berlin, 1998.

\bibitem{BOP}
A.~Braides, R.~Oleinik, and M.~Picerni.
\newblock Higher-order non-local gradient theory of phase-transitions: the
  higher-dimensional case. in preparation.

\bibitem{BT}
A.~Braides and L.~Truskinovsky.
\newblock Asymptotic expansions by {$\Gamma$}-convergence.
\newblock {\em Contin. Mech. Thermodyn.}, 20(1):21--62, 2008.

\bibitem{BV}
A.~Braides and E.~Vitali.
\newblock Fractional approximation of free-discontinuity problems. {I}n
  preparation.

\bibitem{BDS}
G.~C. Brusca, D.~Donati, and M.~Solci.
\newblock Higher-order singular perturbation models for phase transitions.
\newblock {\em SIAM J. Math. Anal.}, 57(3):3146--3170, 2025.

\bibitem{BDT}
G.~C. Brusca, D.~Donati, and C.~Trifone.
\newblock Singular perturbations models in phase transitions for anisotropic
  higher-order materials.
\newblock {\em Calc. Var. Partial Differential Equations}, to appear.

\bibitem{MR527419}
G.~Buttazzo and G.~Dal~Maso.
\newblock {$\Gamma $}-limit of a sequence of nonconvex and non-equi-{L}ipschitz
  integral functionals.
\newblock {\em Ricerche Mat.}, 27(2):235--251, 1978.

\bibitem{CH}
J.~W. Cahn and J.~E. Hilliard.
\newblock Free energy of a nonuniform system. {I}. {I}n\-ter\-fa\-cial free
  energy.
\newblock {\em J. Chem. Phys.}, 28(2):258--267, 1958.

\bibitem{MR3748585}
G.~Dal~Maso, I.~Fonseca, and G.~Leoni.
\newblock Asymptotic analysis of second order nonlocal {C}ahn--{H}illiard-type
  functionals.
\newblock {\em Trans. Amer. Math. Soc.}, 370(4):2785--2823, 2018.

\bibitem{FM}
I.~Fonseca and C.~Mantegazza.
\newblock Second order singular perturbation models for phase transitions.
\newblock {\em SIAM J. Math. Anal.}, 31(5):1121--1143, 2000.

\bibitem{FT}
I.~Fonseca and L.~Tartar.
\newblock The gradient theory of phase transitions for systems with two
  potential wells.
\newblock {\em Proc. Roy. Soc. Edinburgh Sect. A}, 111(1-2):89--102, 1989.

\bibitem{GP-2}
M.~Gobbino and N.~Picenni.
\newblock A quantitative variational analysis of the staircasing phenomenon for
  a second order regularization of the {P}erona-{M}alik functional.
\newblock {\em Trans. Amer. Math. Soc.}, 376(8):5307--5375, 2023.

\bibitem{gurtin}
M.~E. Gurtin.
\newblock Some results and conjectures in the gradient theory of phase
  transitions.
\newblock In {\em Metastability and Incompletely Posed Problems}, pages
  135--146. Springer, 1987.

\bibitem{leofrac}
G.~Leoni.
\newblock {\em A {F}irst {C}ourse in {F}ractional {S}obolev {S}paces}, volume
  229 of {\em Graduate Studies in Mathematics}.
\newblock American Mathematical Society, Providence, RI, 2023.

\bibitem{Ms}
V.~Maz'ya and T.~Shaposhnikova.
\newblock On the {B}ourgain, {B}rezis, and {M}ironescu theorem concerning
  limiting embeddings of fractional {S}obolev spaces.
\newblock {\em J. Funct. Anal.}, 195(2):230--238, 2002.

\bibitem{Modica}
L.~Modica.
\newblock The gradient theory of phase transitions and the minimal interface
  criterion.
\newblock {\em Arch. Ration. Mech. Anal.}, 98:123--142, 1987.

\bibitem{MM}
L.~Modica and S.~Mortola.
\newblock Un esempio di {$\Gamma \sp{-}$}-convergenza.
\newblock {\em Boll. Un. Mat. Ital. B (5)}, 14(1):285--299, 1977.

\bibitem{PaVi}
G.~Palatucci and S.~Vincini.
\newblock Gamma-convergence for one-dimensional nonlocal phase transition
  energies.
\newblock {\em Matematiche (Catania)}, 75(1):195--220, 2020.

\bibitem{Picerni}
M.~Picerni.
\newblock Analysis for non-local phase transitions close to the critical
  exponent {$s=\frac{1}{2}$}.
\newblock {\em Ric. Mat.}, 74(3):1599--1626, 2025.

\bibitem{SV}
O.~Savin and E.~Valdinoci.
\newblock {$\Gamma$}-convergence for nonlocal phase transitions.
\newblock {\em Ann. Inst. H. Poincar\'e{} C Anal. Non Lin\'eaire},
  29(4):479--500, 2012.

\bibitem{Solci}
M.~Solci.
\newblock Higher-order non-local gradient theory of phase-transitions.
\newblock {\em Milan J. Math.}, 93:455--486, 2025.

\bibitem{MR4937467}
M.~Solci.
\newblock Local interpolation techniques for higher-order singular
  perturbations of non-convex functionals: {F}ree-discontinuity problems.
\newblock {\em J. Math. Pures Appl. (9)}, 204:Paper No. 103776, 2025.

\end{thebibliography}

\end{document}